\documentclass[11pt]{amsart}
\usepackage{amssymb,amsmath,latexsym,enumerate, dsfont}
\usepackage{graphicx,verbatim,enumerate,%
	ifpdf,mdwlist}
\usepackage{mathabx}
\usepackage{mathtools}
\usepackage{color}
\ifpdf
\usepackage{hyperref}
\else
\usepackage[hypertex]{hyperref}
\fi
\usepackage{graphicx}
\usepackage{cancel}
\usepackage{float}
\usepackage{enumitem}
\renewcommand{\phi}{\varphi}
\newtheorem{theorem}{Theorem}[section]

\newtheorem{lemma}[theorem]{Lemma}

\newtheorem{corollary}[theorem]{Corollary}
\newtheorem{proposition}[theorem]{Proposition}

\newtheorem{defi}[theorem]{Definition}
\newenvironment{definition}{\begin{defi} \rm}{ \end{defi}}
\newtheorem{exa}[theorem]{Example}

\newtheorem{notice1}{Observation}
\newtheorem{rem}[theorem]{Remark}
\newenvironment{remark}{\begin{rem} \rm}{ \end{rem}}

\DeclareMathOperator{\dom}{domain}

\newcommand{\QA}[1]{(\forall{#1})\,}

\title[Conjunctive reducibilities and completeness]
{Conjunctive reducibilities and completeness}

\author[I.~Chitaia]{Irakli Chitaia}
\address{Department of Mathematics \\
	Ivane Javakhishvili Tbilisi State University\\ Tbilisi 0186, Georgia}
\email{\href{mailto:i.chitaia@gmail.com}{i.chitaia@gmail.com}}

\author[K.M.~Ng]{Keng Meng Ng}
\address{Division of Mathematical Sciences, School of Physical \& Mathematical
	Sciences, College of Science\\
	Nanyang Technological University\\
	Singapore}
\email{\href{mailto:kmng@ntu.edu.sg}{kmng@ntu.edu.sg}}

\author[R.~Omanadze]{Roland Omanadze}
\address{Department of Mathematics \\
	Ivane Javakhishvili Tbilisi State University\\ Tbilisi 0186, Georgia}
\email{\href{mailto:roland.omanadze@tsu.ge}{roland.omanadze@tsu.ge}}

\author[A.~Sorbi]{Andrea Sorbi}
\address{Dipartimento di Ingegneria Informatica e Scienze Matematiche\\
	Universit\`a Degli Studi di Siena\\
	I-53100 Siena, Italy}
\email{\href{mailto:andrea.sorbi@unisi.it}{andrea.sorbi@unisi.it}
}\thanks{The fourth author is a member of GNSAGA}



\begin{document}
\maketitle 

\begin{abstract}
In this article we study the notion of completeness for conjunctive 
reducibilities. We investigate the relationship between $c$-completeness 
and $r$-completeness of computably enumerable (c.e.) sets with respect to 
various strong reducibilities $\le_r$. By using simplicity properties of 
sets, we prove that there exist c.e.\ sets that are simultaneously 
$Q$-complete and $bd$-complete, yet fail to be $c$-complete. Similarly, 
there exist c.e.\ sets that are	simultaneously $Q$-complete and 
$bwtt$-complete (respectively, $btt$-complete)  but not $c$-complete. 
Furthermore, we study two restrictions of $c$-reducibility, namely $c_1$- 
and $c_{1,N}$-reducibility, and show that they are distinct on the c.e.\ 
sets. Nevertheless, we prove that the notions of completeness for $c$, 
$c_1$, and $c_{1,N}$ coincide. 
\end{abstract}

\section{Introduction}

In this paper we study conjunctive reducibility ($c$-reducibility), first 
proposed by Jockusch~\cite{Jockusch1969}, and its subreducibility (known as 
$c_{1,N}$-reducibility), suggested by Bulitko~\cite{Bulitko1992}. Conjunctive 
reducibility is a \emph{positive reducibility}, i.e., a notion of 
computability relative to an external oracle in which the computing agent 
makes use only of positive information about the oracle. In this sense, it is 
a restricted version of enumeration reducibility, although it provides a 
degree structure whose least degree consists exactly of the decidable sets. 
This marks a fundamental difference from the most commonly studied positive 
reducibilities, whose least degree consists instead of the computably 
enumerable (c.e.) sets. 
 
Formally, a set $A$ is \emph{$c$-reducible} to a set $B$ (denoted $A  \le_c 
B$) if there exists a computable function $f$ such that for all $x, y \in 
\omega$: 
\[
x \in A \iff D_{f(x)} \subseteq B:
\]
we say in this case that $f$ is a \emph{$c$-reduction from $A$ to $B$}. 
(Throughout the paper the symbol $\omega$ denotes the set of natural numbers, 
and the symbol $D_u$ denotes the finite set with canonical index $u$.) 
Interest in studying the algebraic structure of the $c$-degrees increased in 
the wake of unexpected results by Dobritsa \cite{Dobritsa} and Belegradek 
\cite{Belegradek}, which revealed that structural properties of the 
$c$-degrees correspond to properties of classes of finitely generated 
subgroups of algebraically closed groups. Degtev \cite{Degtev} proved that 
the upper semilattice $R_c$ of the c.e. $c$-degrees is not distributive and 
is not elementarily equivalent to the structure $R_r$ of the c.e. 
$r$-degrees, with $r\in \{tt, l, p, d\}$ (cf. \cite{Degtev}, to which the 
reader is referred also for the meaning of the acronyms). Moreover, 
Degtev~\cite{Degtev:comparison} proved that every maximal set has minimal 
$c$-degree.  If, in addition, a $c$-reduction $f$ from $A$ to $B$ satisfies 
\[
x \neq y \implies D_{f(x)} \cap D_{f(y)} = \emptyset
\] 
then we say that \emph{$A$ is $c_1$-reducible to $B$} (notation $A \leq_{c_1} 
B$) and $f$ is a \emph{$c_1$-reduction from $A$ to $B$}. This reducibility 
can naturally be regarded as the injective version of $c$-reducibility. The 
current state of knowledge about conjunctive reducibilities, and in 
particular the algebraic structures of $c$- and $c_1$-c.e.\ degrees, can be 
found in~\cite{ChitaiaOmanadze2022, ChitaiaOmanadzeSorbi2021, 
ChitaiaOmanadzeSorbi2025}. 

A further restriction of $c$-reducibility, more precisely of 
$c_1$-reducibility, is defined as follows. A set $A$ is 
\emph{$c_{1,N}$-reducible} to a set $B$ (denoted $A  \le_{c_{1,N}} B$) if 
there is a $c_1$-reduction $f$ from $A$ to $B$ such that the set 
\[ 
R_f=\bigcup_{x \in \omega} D_{f(x)}
\]
is computable. In this case we say that $f$ is a \emph{$c_{1,N}$-reduction 
from $A$ to $B$}. It is obvious that $\leq_{1,N} \subseteq \leq_{c_1}$. This 
restriction on $c_1$-reducibility, according to which only those 
$c_1$-reductions whose set of all possible queries is decidable are allowed 
to perform oracle computations, is in line with similar restrictions proposed 
for some classical reducibilities in Rogers' 
book~\cite[Exercise~7.34]{Rogers1987}. 

This article is devoted to the study of completeness notions for conjunctive 
reducibilities. Section~2 contains some remarks on the relationships between 
$c$-completeness and $r$-completeness of c.e.\ sets, where $\leq_r$ ranges 
over various strong reducibilities. In particular, by using simplicity 
properties of sets, we show (Theorem~\ref{theorem-1.6}) that there exists a 
c.e.\ set which is simultaneously $Q$-complete and $bd$-complete, but not 
$c$-complete; moreover, Corollary~\ref{cor:relationsh} shows that there 
exists a c.e.\ set which is simultaneously $Q$-complete and $bwtt$- ($btt$-) 
complete but not $c$-complete. In Section~3, we prove that on c.e.\ sets 
$c_{1,N}$-reducibility is strictly stronger than $c_1$-reducibility. Finally, 
in Section~4 we show that the notions of $c$-, $c_1$-, and 
$c_{1,N}$-completeness coincide. 

Unless otherwise specified, our notations and terminology concerning 
computability theory are standard and can be found in the classical textbooks 
\cite{Rogers1987} and \cite{Soare:Book}. The complement $\omega 
\smallsetminus X$ of a set $X \subseteq \omega$ will be denoted by 
$\overline{X}$, and its cardinality by $\left|X\right|$. For a given 
reducibility $r$, and a c.e. set $A$, we say that $A$ is $r$-complete if 
$B\leq_r A$ for every c.e. set $B$. 

\section{Relationships between $c$-, $Q$- and $bd$-complete sets }

In this section, we prove that there exists a c.e.\ set which is 
simultaneously $Q$-complete and $bd$-complete, but not $c$-complete. Nowhere 
simple sets, introduced by Shore~\cite{Shore1978}, are an important tool for 
the investigation carried out in this section. 

\begin{definition}
A set $A$ is \emph{nowhere simple} if for every c.e.\ set $B$ with $B 
\setminus A$ infinite, there is an infinite c.e.\ set $W \subseteq B 
\setminus A$.
\end{definition} 

Shore \cite[Theorem~4]{Shore1978} proved that every c.e. degree contains a 
nowhere simple set.

The following theorem, which is of independent interest, is useful for 
proving the main result of the section. 

\begin{theorem}\label{theorem-1.1}
Let $A$ be a simple set, $B$ be an arbitrary set, $C$ be a nowhere simple set 
and 
\[ 
A \leq_c B \oplus C. 
\]
Then $A \leq_c B$.
\end{theorem}

\begin{proof}
We first prove the following lemma, along the lines of a corresponding result 
for $Q$-reducibility, proved in \cite[Lemma~3.3]{Omanadze2019}. 
	
\begin{lemma}\label{lemma-1.2}
Let $A$ be a simple set, $B$ be an arbitrary set, $C$ be a nowhere simple set 
and $A \leq_c B \oplus C$ via $f$. As a consequence, for every c.e.\ set $W$ 
for which $W \setminus A$ is infinite, we have: 
\begin{enumerate}
\item $\left| \{ 2x+1 : 2x+1 \in \bigcup_{y \in W} D_{f(y)} \} \cap 
    \overline{B \oplus C} \right| < \infty,$ 
\item $\left| \{ 2x : 2x \in \bigcup_{y \in W} D_{f(y)} \} 
\cap \overline{B 
    \oplus C} \right| = \infty.$ 
\end{enumerate}
\end{lemma}

\begin{proof}[Proof of Lemma~\ref{lemma-1.2}]
First of all, notice that Claim~(2) obviously follows from Claim~(1). So, it 
is enough to prove Claim~(1). Let $U=\bigcup_{y \in W} D_{f(y)}$, and let $V= 
U \cap (2\omega+1)$. Throughout the proof of this lemma, for any set of 
numbers $X$, let $X^-=\{x: 2x+1 \in X\}$ and $X^+=\{2x+1: x \in X\}$. 
Assuming that the conditions of the lemma are satisfied, we first show that 
\[ 
\left|U \cap \overline{B \oplus C}\right| = \infty. 
\]
Indeed, if this set were finite, then the set
\[ 
\widetilde{W} = \left\{ x : D_{f(x)} \cap \left(U \cap 
\overline{B \oplus C} \right) \neq \varnothing \right\} 
\]
would be c.e.\ and $W \setminus A \subseteq \widetilde{W}$ while 
$\widetilde{W} \cap A = \varnothing$, providing an infinite c.e.\ subset of 
$\overline{A}$, which contradicts the simplicity of $A$. 

Assume that $\left|V\setminus B\oplus C\right|=\infty.$ Then $\left|V^- 
\setminus C\right| = \infty$. Since $C$ is nowhere simple, there exists an 
infinite c.e.\ set $V_1$ such that 
$$
V_1 \subseteq V^- \setminus C.
$$ 
Then  $V_1^+$ is infinite,  $V_1^+ \subseteq \overline{B \oplus C}$ and 
$V_1^+ \subseteq V$. 

Consider 
$$
V_2 = \{ x : D_{f(x)} \cap V_1^+ \neq \varnothing \}. 
$$ 
Then $V_2$ is a c.e. set such that $V_2 \subseteq \overline{A}$ (as $x \in A$ 
implies $D_{f{(x)}} \subseteq 2\omega$). Moreover, $V_2$ is infinite since 
$V_1^+ \subseteq V$ is infinite and one need infinitely many $x$ such that 
$\bigcup_x D_{f(x)}$ covers $V_1^+$. Such a c.e. set contradicts the 
assumption that $A$ is simple.
\end{proof}

Let us now go back to the proof of Theorem~\ref{theorem-1.1}. Recall that $f$ 
is a computable function such that 
$$
\QA{x}\left[ x \in A \iff D_{f(x)} \subseteq B \oplus C\right]. 
$$ 
		Let 
$$
R = (2\omega +1) \cap \left( \bigcup_{y \in \omega} 
D_{f(y)}\right) \cap \overline{B \oplus  C}. 
$$ 
By (1) of Lemma~\ref{lemma-1.2}, $\left|R\right| < \infty$. Let 
\begin{align*}
D_{f_1(x)} &= \{ y : 2y \in D_{f(x)} \}\\
R_1 &= \{ x : D_{f(x)} \cap R \neq \varnothing \}.
\end{align*}

Since $R_1$ is an infinite c.e. subset of $\overline{A}$, by simplicity of 
$A$ we have that $\left|R_1\right| < \infty$. Let $f_2$ be a computable 
function such that 
\[
D_{f_2(x)} = \begin{cases} 
D_{f_1(x)}, & \text{if } x \notin R_1 \\ 
\{a\}, & \text{if } x \in R_1, 
\end{cases}
\]
where $a$ is a fixed element of $\overline{B}$.

If $x \in A$, then $D_{f(x)} \subseteq B \oplus C$. Since $A \cap R_1 = 
\varnothing$, $D_{f_1(x)} \subseteq B$. If $x \notin A$, then either $x \in 
R_1$ (so $D_{f_2(x)} = \{a\} \nsubseteq B$) or $x \notin R_1$ (so $D_{f(x)} 
\nsubseteq B \oplus C$ being $x \notin A$, and $D_{f(x)} \cap R = 
\varnothing$, implying $D_{f_2(x)}=D_{f_1(x)} \nsubseteq B$). In all cases, 
\[
\QA{x}\left[ x \in A \iff D_{f_2(x)} \subseteq B\right].
\]
Thus $A \le_c B$.
\end{proof}

For the following lemma we need the following definition. 

\begin{definition}\label{def:Q-Q1}
A set $A$ is said to be \emph{$Q$-reducible to a set $B$} (in symbols: $A 
\le_Q B$) if there exist a computable function $f$ such that for all $x$: 
\[ x \in A \iff W_{f(x)} \subseteq B. \]
If in addition $f$ satisfies $W_{f(x)} \cap W_{f(y)}=\emptyset$ whenever 
$x\ne y$, then we say that \emph{$A$ is $Q_1$-reducible to $B$} (notation $A 
\leq_{Q_1} B$). 
\end{definition} 
	
\begin{lemma}\label{lemma-1.3}
Let $A_0$ and $A_1$ be c.e.\ sets. Then $A_0 \oplus A_1$ is a nowhere simple 
set if and only if $A_0$ and $A_1$ both are nowhere simple. 
\end{lemma}

\begin{proof}
In \cite{Omanadze2019} it is proved that if $A_0, A_1$ are nowhere simple, 
then $A_0 \oplus A_1$ is nowhere simple. Conversely, if $A_0 \oplus A_1$ is 
nowhere simple, since $A_0 \le_Q A_0 \oplus A_1$ and $A_1 \le_Q A_0 \oplus 
A_1$, it follows by Proposition 3.5 in \cite{Omanadze2019} that $A_0$ and 
$A_1$ are nowhere simple. 
\end{proof}

\begin{definition}
A set $A$ is said to be \emph{$bd$-reducible to a set $B$} (in symbols: $A 
\le_{bd} B$) if there exist a computable function $f$ and $n \in \omega$ such 
that for all $x$: 
\[
x \in A \iff D_{f(x)} \cap B \neq \varnothing 
\] 
and 
\[
\left|D_{f(x)}\right| < n. 
\]
\end{definition}
(The acronym $bd$ stands for \emph{bounded disjunctive reducible}: A set $A$ 
is \emph{disjunctive reducible} to a set $B$ if there exists a computable 
function $f$ such that $x \in A \iff D_{f(x)} \cap B \neq \varnothing$, for 
every $x$.) 
	
\begin{proposition}\label{proposition-1.4}
There exists a $bd$-complete nowhere simple set.
\end{proposition}

\begin{proof}
Let $A$ be a c.e.\ $bd$-complete set. By \cite[Theorem~2]{Shore1978}, $A$ can 
be split into two disjoint nowhere simple sets $A_0, A_1$. Clearly $A 
\equiv_{bd} A_0 \oplus A_1$, and by the previous Lemma, $A_0 \oplus A_1$ is 
nowhere simple. 
\end{proof}

For the following proposition, and its proof, we recall the following two 
definitions. 

\begin{definition}
A set of numbers $A$ is \emph{semicomputable} if there is a two-variable 
computable function $f$ such that for every $x,y \in \omega$, we have 
$f(x,y)\in \{x,y\}$ and if $\{x,y\} \cap A\neq \emptyset$ then $f(x,y) \in 
A$. 
\end{definition}

And

\begin{definition}
A c.e. set $A$ is \emph{hypersimple} if $\overline{A}$ is infinite, and for 
every computable function $f$, such that $D_{f(x)}\cap D_{f(y)} =\emptyset$ 
whenever $x \ne y$, we have that $D_{f(x)} \subseteq A$ for at least one $x$. 
\end{definition}
(Any collection of finite sets $\{D_{g(x)}\}_x$, where $g$ is computable and 
the sets are pairwise disjoint is called a \emph{disjoint strong array}.) 

\begin{proposition}\label{proposition-1.5}
There exists a $Q$-complete hypersimple set.
\end{proposition}

\begin{proof}
Dekker (see Theorem 9.XVI in \cite{Rogers1987}) proved there exists a 
$T$-complete semicomputable hypersimple set. Marchenkov 
\cite[Lemma~1]{Marchenkov1976} showed that if $B$ is a semicomputable c.e. 
set such that $B \neq \varnothing, \omega$ and $A$ is a c.e. set then $A 
\le_T B \Rightarrow A \le_Q B$. Thus, $Q$-complete hypersimple sets do exist. 
\end{proof}

As an application of Theorem~\ref{theorem-1.1}, 
Proposition~\ref{proposition-1.4} and Proposition~\ref{proposition-1.5}, we 
get:
\begin{theorem}\label{theorem-1.6}
There exists a c.e.\ set which is simultaneously $Q$-complete and 
$bd$-complete, but not $c$-complete. 
\end{theorem}

\begin{proof}
Let $A$ be a $Q$-complete hypersimple set (such set exists by 
Proposition~\ref{proposition-1.5}), and let $B$ be a $bd$-complete nowhere 
simple set (such a set exists by Proposition~\ref{proposition-1.4}). Consider 
$A \oplus B$. This set is $Q$-complete and $bd$-complete. Suppose $A \oplus 
B$ is $c$-complete. Let $S$ be a $c$-complete simple set (such a simple set 
exists by Theorem 8.VIII in \cite{Rogers1987}). Then $S \le_c A \oplus B$. By 
Theorem~\ref{theorem-1.1}, $S \le_c A$, so $A$ is $c$-complete. However, no 
hypersimple set can be $wtt$-complete, let alone $c$-complete 
(\cite{Jockusch1969}). Contradiction. 
\end{proof}

\begin{definition}
A set $A$ is \emph{weak truth-table reducible} ($wtt$-reducible) to a set $B$ 
(in symbols: $A \leq_{wtt} B$) \cite{Rogers1987}, if $(\exists z)[\chi_A = 
\varphi_z^B]$ and there is a computable $f$ such that, for every $x$, the 
finite set $D_{f(x)}$ consists exactly of the numbers $y$ queried during the 
computation $\varphi_z^B(x)$, i.e., those numbers for which the oracle is 
asked to determine the value $B(y)$. Here, $\chi_A$ denotes the 
characteristic function of $A$. If, in addition, there exists $n$ such that 
$\left|D_{f(x)}\right| \leq n$ for all $x$, then $A \leq_{bwtt} B$. The 
notion of $bwtt$-reducibility was introduced by Lachlan \cite{Lachlan1975}. 
\end{definition} 

Lachlan \cite{Lachlan1975} proved that $bwtt$-complete sets are 
$btt$-complete (even $bd$-complete). Therefore, from 
Theorem~\ref{theorem-1.6}, we have: 

\begin{corollary}\label{cor:relationsh}
There exists a c.e.\ set which is simultaneously $Q$-complete and $bwtt$- 
($btt$-) complete but not $c$-complete. 
\end{corollary}

\begin{proof}
Immediate.
\end{proof}

\section{On the difference between $c_1$-reducibility and 
$c_{1,N}$-reducibility on c.e. sets}

It is well known that, on c.e. sets, $c$-reducibility and $c_1$-reducibility 
do not coincide. In fact, there are c.e. $c$-degrees which contain infinitely 
many $c_1$-degrees (see for instance \cite{ChitaiaOmanadze2022},  
\cite{ChitaiaOmanadzeSorbi2021} and \cite{ChitaiaOmanadzeSorbi2025}), and the 
poset of c.e. $c$-degrees is not elementarily equivalent to the poset of 
$c_1$-degrees (see~\cite{ChitaiaOmanadze2022}). In this section we show that
$c_1$-reducibility and $c_{1,N}$-reducibility do not coincide on c.e. sets. 

\begin{theorem}\label{theorem1x}
Let $A \le_{Q_1} B$, where $A$ and $B$ are c.e. sets. Then there exists a 
c.e. set $C$ such that: 
\[ 
A \le_{Q_{1,N}} C \le_1 B. 
\]
\end{theorem}

\begin{proof}
Assume that $A$ and $B$ are c.e. sets and $A \le_{Q_1} B$ via $f$. Let $g$ be 
a one-to-one computable function such that 
\[ 
\text{rng}(g) = \bigcup_{x \in \omega} W_{f(x)}. 
\]
Consider the c.e. set $C = g^{-1}(B)$ and let $h$ be a computable function 
such that for all $x$, 
\[ W_{h(x)} = g^{-1}(W_{f(x)}). \]

Then, for all $x, y$:
\begin{itemize}
\item $x \in A \Leftrightarrow W_{h(x)} \subseteq C$,
\item $x \neq y \Rightarrow W_{h(x)} \cap W_{h(y)} = \emptyset$,
\item The set $\bigcup_{x \in \omega} W_{h(x)}$ is computable (in fact, 
    $\bigcup_{x \in \omega} W_{h(x)} = \omega$). 
\end{itemize}
Furthermore, $C \le_1 B$ via $g$. Thus, $A \le_{Q_{1,N}} C \le_1 B$.
\end{proof}

The following Proposition is a direct corollary of Proposition 3.1 of 
\cite{omanadze2015} 

\begin{proposition}\label{proposition}	
Let $A$ and $B$ be c.e. sets. If $B$ is an $r$-maximal set and $A 
\le_{Q_{1,N}} B$, then $A$ is an $r$-maximal set. 
\end{proposition}

\begin{proof}
Immediate from Proposition 3.1 of \cite{omanadze2015}.
\end{proof}

\begin{theorem}\label{theorem2x}
There exist c.e. sets $A, B, C$ such that $B$ is $r$-maximal, $A$ is not 
$r$-maximal, $C$ is not $r$-maximal, and $A \le_{Q_{1,N}} C <_1 B$. 
\end{theorem}

\begin{proof}
Let $B$ be a $Q_1$-complete $r$-maximal set (such a set exists by Corollary 2 
of \cite{omanadze2002} and the fact that $Q$-completeness and 
$Q_1$-completeness coincide \cite{GillMorris1974}) and let $A$ be a 
non-$r$-maximal c.e. set. Then $A \le_{Q_1} B$.  By Theorem~\ref{theorem1x}, 
there exists a c.e. set $C$ such that $A \le_{Q_{1,N}} C \le_1 B$. Since $A$ 
is not $r$-maximal and $A \le_{Q_{1,N}} C$, the set $C$ cannot be $r$-maximal 
by Proposition~\ref{proposition}. Since $C$ is not $r$-maximal , but $B$ is 
$r$-maximal, they are not recursively isomorphic, which means that 
$B\nleq_1C.$ Thus $A \le_{Q_{1,N}} C <_1 B.$ 
\end{proof}

\begin{theorem}\label{theorem3x}
The reducibilities $c_1$ and $c_{1,N}$  do not coincide on the c.e. sets.
\end{theorem}

\begin{proof}
By Theorem~\ref{theorem2x}, there exist c.e. sets $C$ and $B$ such that $C$ 
is not $r$-maximal, $B$ is $r$-maximal, and $C <_1 B$. Since $1$-reducibility 
implies $c_1$-reducibility, it follows that $C \le_{c_1} B$. On the other 
hand, by Proposition~\ref{proposition}, since $B$ is $r$-maximal and $C$ is 
not $r$-maximal, we have that $C \not\le_{c_{1,N}} B$. Therefore, there is a 
clear difference between $c_1$ and $c_{1,N}$ on the class of c.e. sets. 
\end{proof}

\section{$c_{1}$-completeness and $c_{1,N}$-completeness coincide.}

In this section we show that the notions of $c$-completeness, 
$c_1$-completeness and $c_{1,N}$-completeness coincide. Since we know already 
(\cite{Omanadze1976}) that $c$-completeness implies $c_1$ completeness, to 
prove our claim it is enough to show (see Theorem~\ref{thm:from-c1-to-c1N}) 
that $c_1$-completeness implies $c_{1,N}$-completeness. However, we prefer to 
precede Theorem~\ref{thm:from-c1-to-c1N} with a new proof (namely, the proof 
of Lemma~\ref{lem:from-c-to-c1}) of the fact that $c$-completeness implies 
$c_1$-completeness. We do this mainly for pedagogical reasons. First, the 
arguments in the two cases are similar and complement each other. Moreover, 
as far as we know, our proof of Lemma~\ref{lem:from-c-to-c1} gives a novel 
derivation of the fact that $c$-completeness implies $c_1$-completeness, 
introducing a technique that may be adaptable to new similar contexts in 
which, for a given reducibility $r$ and its ``injective'' version $r_1$, one 
aims to prove that $r$-completeness implies $r_1$-completeness. 

\begin{lemma}\cite{Omanadze1976}\label{lem:from-c-to-c1}
If $A$ is $c$-complete then $A$ is $c_1$-complete.
\end{lemma}

\begin{proof}
Let $A$ be a given c.e. $c$-complete set, and let $B$ be any c.e. set: we 
want to show tat $B \leq_{c_1} A$. 

Let us construct a family $\{V_e: e \in \omega\}$ of uniformly c.e. sets 
(i.e., $V_e=W_{k(e)}$ for some computable function $k$) and consider 
\[
W=\bigoplus_{e \in \omega} V_e\left( = \bigcup_e \{e\}\times V_e\right).
\] 
Since $W$ is c.e. and $A$ is $c$-complete, we have that $W\leq_c A$. For each 
$e$, we describe a strategy $\mathcal{I}_e$ which uniformly builds a c.e set 
$V_e$ and a corresponding partial (in fact finite or total) computable 
function $\widehat{f}_e$, such that if $\phi_e$ is a $c$-reduction from $W$ 
to $A$ then $\widehat{f}_e$ is a $c_1$-reduction from $B$ to $A$.  

\begin{remark}
We will work with computable approximations $\{\phi_{e,s}\}_{e,s\in \omega}$ 
to the partial computable functions, such that $\phi_{e,0}(x)$ is undefined 
for every $x$: see \cite[p.17]{Soare:Book}. To simplify notation, we use the 
symbol $f_e$ to denote the partial function $f_e(y)=\phi_e(\langle e, y 
\rangle)$, and accordingly $f_{e,s}(y)=\phi_{e,s}(\langle e, y \rangle)$. 
Notice that if $\phi_e$ is a $c$-reduction from $W$ to $A$, then $f_e$ is a 
$c$-reduction from $V_e$ to $A$. We will also use computable finite 
approximations  $\{A_s\}_{s \in \omega}$ and $\{B_s\}_{s \in \omega}$ to the 
c.e. sets $A,B$ (in the usual meaning of \cite{Soare:Book}). Sets of numbers 
will often be identified with their characteristic functions. Therefore, if 
$X$ is a set of numbers and $y$ is a number, then $X(y)$ denotes the value of 
the characteristic function of $X$ at $y$. 
\end{remark} 

\paragraph{\textbf{The strategy $\mathcal{I}_e$ and the definitions of $V_e$ 
and $\widehat{f}_e$}.}

The strategy $\mathcal{I}_e$ builds a c.e. set $V_e$, under the assumption 
that $\phi_e$ is a $c$-reduction from $W$ to $A$ and thus $f_e$ is a 
$c$-reduction from $V_e$ to $A$. By suitably coding $B$ into $V_e$, we try to 
exploit this $c$-reduction to obtain a $c_1$-reduction $\widehat{f}_e: V_e 
\leq_{c_1} A$, so that by the coding of $B$ into $V_e$ we achieve $B 
\leq_{c_1} A$. If at some stage we see that our attempts at making 
$\widehat{f}_e$ a suitable $c_1$-reduction fail, then our definition of $V_e$ 
will force $\phi_e$ not to be a $c$-reduction from $W$ to $A$. In this case 
we say that we \emph{abandon} the strategy $\mathcal{I}_e$, and we never 
again modify the parameters of the strategy. Since in the end there must be 
some $e$ such that $\phi_e$ is a $c$-reduction from $W$ to $A$, then, for 
this $e$, we have that $V_e \leq_{c_1} A$ and thus $B \leq_{c_1} A$. 

The strategy $\mathcal{I}_e$ works with three parameters, namely $a_e(s)$ (a 
number), $\widehat{f}_{e,s}$ (a finite function), and $V_{e,s}$ (a finite 
set). The parameters $\widehat{f}_{e,s}$ and $V_{e,s}$ will be finite 
approximations in the usual way to a partial (in fact finite or total) 
computable function $\widehat{f}_e$ and a c.e.\ set $V_e$. If $\mathcal{I}_e$ 
has not as yet been abandoned by stage $s$, then at stage $s+1$ the strategy 
updates its parameters for which the following inductive assumptions hold: 

\begin{enumerate}
\item $\dom(\widehat{f}_{e,s})=[0,a_e(s))$, and $[0,a_e(s)) \subseteq 
    \dom(f_{e,s})$; 
\item $V_{e,s}(y)=B_s(y)$, for every $y< a_e(s)$. 
\end{enumerate}
At this stage $s+1$ we also decide whether or not to abandon $\mathcal{I}_e$. 

\medskip
Stage $0$) Let $a_e(0)=0$ and $V_{e,0}=\widehat{f}_{e,0}=\emptyset$.

\medskip 

Stage $s+1$)  If we have already decided at a previous stage to abandon 
$\mathcal{I}_e$, then do nothing and  leave the values of the parameters for 
$\mathcal{I}_e$ at $s+1$ to be the same as at stage $s$. Otherwise, 
\begin{enumerate}

\item If we see that there exists some $y< a_e(s)$ such that $y \notin 
    V_{e,s}$ and $D_{\widehat{f}_{e,s}(y)} \subseteq A_s$, then put all $z 
    < y$ into $V_e$ by letting  $V_{e,s+1}= V_{e,s} \cup [0,y)$, and 
    abandon $\mathcal{I}_e$ (hence, $y \notin V_e$).

\item Otherwise, if $f_{e,s}(a_e(s))$ is still undefined then go to next 
    stage, and leave the values of the parameters at $s+1$ to be the same 
    as at stage $s$. 

\item Otherwise, define $a_e(s+1)=a_e(s)+1$, and distinguish the following 
    two cases: 

\medskip

Case~(i): $D_{f_{e}(a_e(s))} \subseteq \bigcup_{y < a_e(s)} D_{f_{e}(y)}$. 
In this case, put each $y < a_e(s)$ into $V_e$ by letting $V_{e,s+1}= 
V_{e,s} \cup [0,a_e(s))$, abandon $\mathcal{I}_e$ (hence, $a_e(s) \notin 
V_e$).

\medskip Case~(ii): Otherwise, let $\widehat{f}_e(a_e(s))$ be such that
\[
D_{\widehat{f}_e(a_e(s))} = D_{f_e(a_e(s))} \smallsetminus 
\bigcup_{y < a_e(s)} D_{f_e}(y).
\]

\end{enumerate}
(Updating of $V_e$.) If we have not abandoned $\mathcal{I}_e$, then at the 
end of stage $s+1$ update $V_{e,s+1}(y)=B_s(y)$ for every $y< a_e(s+1)$. 

\medskip

\paragraph{\textbf{Verification}.} 

Since $W \leq_c A$, let $e$ be such that $\phi_e$ is a total computable 
function which $c$-reduces $W$ to $A$. We claim in this case that the 
strategy $\mathcal{I}_e$ is never abandoned. Otherwise there is some $y$ 
(coming from Case (1), or $y=a_e(s)$, for some $s$, coming from Case (3(i)) 
of the construction), such that $\langle e, y \rangle \notin W$ (since $y 
\notin V_e$) but $D_{\phi_e(\langle e, y\rangle)} \subseteq A$, contradicting 
that $f_e$ is a $c$-reduction. Indeed, to see that $D_{\phi_e(\langle e, 
y\rangle}) \subseteq A$, we first notice that in both cases $\bigcup_{z < y} 
D_{\phi_e(\langle e, z\rangle)} \subseteq A$ being $[0,y) \subseteq V_e$ and 
thus  $\{e\} \times [0,y) \subseteq W$. On the other hand, in Case (1) since 
$D_{\widehat{f}_e(y)}= D_{f_e(y)} \smallsetminus \bigcup_{z < y} 
D_{f_e(z))}$, we have 
\[ D_{\phi_e(\langle e, y\rangle)}=D_{f_e(y)} \subseteq \bigcup_{z < y} 
D_{f_e(z))} \cup D_{\widehat{f}_e(y} \subseteq A. 
\]
In Case (3i) we have 
\[
D_{\phi_e(\langle e, y\rangle)}= D_{f_e(y)} \subseteq 
\bigcup_{z < y} D_{f_e(z))} \subseteq A.
\] 

We now show that $\widehat{f}_e$ is a $c$-reduction from $B$ to $A$. If $y\in 
B$ then $y \in V_e$ by the step-by-step updating of $V_{e,s}$. Therefore the 
pair $\langle e, y\rangle$ lies in $W$ and thus $D_{\phi_e(\langle e, y 
\rangle)}=D_{f_e(y)} \subseteq A$, but then $D_{\widehat{f}_e(y)}\subseteq A$ 
since $D_{\widehat{f}_e(y)} \subseteq D_{f_e(y)}$. If now $y\notin B$ then $y 
\notin V_e$ since $\mathcal{I}_e$ cannot force $y \in V_e$ without being 
abandoned, but $\mathcal{I}_e$ is never abandoned as we have already seen. On 
the other hand, it cannot be $D_{\widehat{f}_e(y)} \subseteq A$, otherwise at 
some stage  $y$ would be a candidate for abandoning $\mathcal{I}_e$. 

In the end we need to show that $\widehat{f}_e$ is in fact a $c_1$-reduction, 
i.e.\ it satisfies $D_{\widehat{f}_e(x)} \cap D_{\widehat{f}_e(y)}=\emptyset$ 
if $x \ne y$. But this follows from the fact that by Case ii) of the 
construction, for every $x$ we have  $D_{\widehat{f}_e(x)} \subseteq 
D_{f_e(x)}$, but $D_{\widehat{f}_e(x)} \cap \bigcup_{y<x}D_{f_e(y))} = 
\emptyset$, and thus $D_{\widehat{f}_e(x)}\cap D_{\widehat{f}_e(y)} = 
\emptyset$ if $x>y$. 
\end{proof}

\begin{theorem}\label{thm:from-c1-to-c1N}
If $A$ is $c_1$-complete then it is $c_{1,N}$-complete.
\end{theorem}

\begin{proof}
Let $A$ be a $c_1$-complete c.e. set, and let $B$ be any c.e. set. We keep 
the notations and the terminology of the previous proof.

Uniformly in $e$ we will define a c.e. set $V_e$, and a partial (in fact, 
finite or total) computable function $\widehat{f}_e$. 

\paragraph{\textbf{The strategy $\mathcal{I}_e$ and the definitions of 
$V_{e}$ and $\widehat{f}_e$}.}

The strategy $\mathcal{I}_e$ defining $V_{e}$ and $\widehat{f}_e$ is similar 
to the one in the proof of Lemma~\ref{lem:from-c-to-c1}. A slight 
complication is the delayed coding of $B$ into $V_e$. Suppose that the 
strategy has defined $\widehat{f}_e(y)$ for all $y<x$, and has already coded 
into $V_e$ the values $B(y)$ for all $y<x$. In order to define 
$\widehat{f}_e(x)$ and code $B(x)$ into $V_e$, the strategy waits for a 
number $a$ such that $\min(D_{f_e(a)})> \max \left(\bigcup_{y< x} 
D_{\widehat{f}_e(y)}\right)$, defines $\widehat{f}_e(x)= f_e(a)$, and makes 
the number $a$ to be the coding location $b_e(x)$ of $x$ in $V_e$. If the 
coding is more complicated, the rest of the construction is simpler than the 
one in the proof of Lemma~\ref{lem:from-c-to-c1}, since, as we will see, 
nothing goes wrong and there is ever no need to abandon $\mathcal{I}_e$, and 
make $b_e(y) \in V_e$ if $y \notin B$. 

\medskip
At stage $s$ we have this time the following parameters for $\mathcal{I}_e$: 
$a_e(s)$, $x_e(s), V_{e,s}, b_{e,s}$. As we see, with respect to the proof of 
the previous lemma, we have two additional parameters, namely a number 
$x_e(s)\leq a_e(s)$, and a finite function $b_{e,s}$.   

\medskip
Stage $0$) Let $V_{e,0}=\widehat{f}_{e,0}= b_{e,0}=\emptyset$, 
$x_e(0)=a_e(0)=0$. 

\medskip 

Stage $s+1$) The strategy updates its parameters for which the following 
assumptions inductively hold: 

\begin{enumerate}
\item $\dom(\widehat{f}_{e,s})=\dom(b_e(s))=[0, x_e(s))$ and $[0, a_e(s)) 
    \subseteq  \dom(f_{e,s})$; 
\item $V_{e,s}(b_{e,s}(y))=B_s(y)$, for every $y< x_e(s)$, and 
    $V_{e,s}(x)=0$ otherwise. 
\end{enumerate} We distinguish the following cases: 
\begin{enumerate}

\item If $f_{e,s}(a_e(s))$ is still undefined then go to next stage, and 
    leave the values of the parameters for $\mathcal{I}_e$ at $s+1$ to be 
    the same as at stage $s$. 

\item If $f_{e,s}(a_e(s))$ converges, and  
\begin{equation}\label{eqn:1}
\tag{*}
\min(D_{f_e(a_e(s))})> \max \left(\bigcup_{y< x_e(s)} 
D_{\widehat{f}_e(y)}\right)
\end{equation}
then let 
\[
\widehat{f}_e(x_e(s))=f_e(a_e(s)),
\] 
and define $x_e(s+1)=x_e(s)+1$, $a_e(s+1)=a_e(s)+1$. Code $B_s(x)$ (where 
for simplicity we write $x=x_e(s)$) into $V_e$, by defining 
$V_{e,s+1}(a_e(s))=B_s(x)$: we thus view $a_e(s)$ as the coding location of 
$B(x)$ in $V_e$, and we record this by defining $b_{e,s}(x)=a_e(s)$. 

\item Otherwise let $a_e(s+1)=a_e(s)+1$, leave the values of the other 
    parameters for $\mathcal{I}_e$ at $s+1$ to be the same as at stage $s$ 
    and go to the next stage. 
\end{enumerate} 

(Updating of $V_e$.) Finally, update $V_{e,s+1}(b_{e,s}(y))=B_s(y)$, for 
every $y< x_e(s)$. 

\medskip

\paragraph{\textbf{Verification}.} 
Since $W \leq_{c_1} A$, let $\phi_e$ be a total computable function which 
$c_1$-reduces $W$ to $A$. By injectivity of the $c_1$-reduction $\phi_e$ 
(i.e. $D_{\phi_e(x)} \cap D_{\phi_e(y)}=\emptyset$ if $x \ne y$) sooner or 
later our waits for desired numbers $a$ satisfying (\ref{eqn:1}) will  end. 
Therefore $\widehat{f}_e$ and $b_e$ are total.

If $y\in B$ then $b_e(y) \in V_e$ and therefore the pair $\langle e, b_e(y) 
\rangle$ lies in $W$, so that $D_{f_e(b_e(y))} \subseteq A$. Since 
$D_{\widehat{f}_e(y)} = D_{f_e(b_e(y))}$, we have $D_{\widehat{f}_e(y)} 
\subseteq A$. 

Suppose now that $y\notin B$. In this case we have that $D_{\widehat{f}_e(y)} 
\not \subseteq A$, since $D_{\widehat{f}_e(y)}= D_{f_e(b_e(y))}$ and $b_e(y) 
\notin V_e$, so that $D_{f_e(b_e(y))} \nsubseteq A$. 

It follows that $\widehat{f}_e$ is a total computable function which 
$c$-reduces $B$ to $A$. Injectivity of the reduction (making $\widehat{f}_e$ 
a $c_1$-reduction) is immediate since it follows by (\ref{eqn:1}) that
\[
x>y \Rightarrow \min(D_{\widehat{f}(x)}) > \max(D_{\widehat{f}(y)})
\]
since $D_{\widehat{f}_e(y)}$ is one of the addenda of the union in 
(\ref{eqn:1}). 

Finally we show that the range of $\widehat{f}_e$ is decidable. To decide 
whether $z\in \bigcup_x D_{\widehat{f}_e(x)}$, compute the least $x$ such 
that $z < \min(D_{\widehat{f}_e(x)})$. Then $z\in \bigcup_x 
D_{\widehat{f}_e(x)}$ if and only if $z \in \bigcup_{y < x} 
D_{\widehat{f}_e(y)}$.   

\end{proof}	

\bibliographystyle{amsplain}
\bibliography{completeness-c-c1N}

\end{document}